\documentclass[intlimits,12pt,reqno]{amsart}
\usepackage{latexsym,amsthm,amsmath,amssymb,mathrsfs,mathtools}
\usepackage[T1]{fontenc}
\usepackage[utf8]{inputenc}
\usepackage[mathscr]{eucal}

\usepackage{tikz}
\usetikzlibrary{arrows,calc,patterns,cd}
\usepackage{wrapfig}

\def\mvint_#1{\mathchoice
          {\mathop{\vrule width 6pt height 3 pt depth -2.5pt
                  \kern -9pt \intop}\limits_{\kern -3pt #1}}%
          {\mathop{\vrule width 5pt height 3 pt depth -2.6pt
                  \kern -6pt \intop}\nolimits_{#1}}%
          {\mathop{\vrule width 5pt height 3 pt depth -2.6pt
                  \kern -6pt \intop}\nolimits_{#1}}%
          {\mathop{\vrule width 5pt height 3 pt depth -2.6pt
                  \kern -6pt \intop}\nolimits_{#1}}}
\newcommand{\MM}{\mathcal M}

\newcommand{\bbbr}{\mathbb R}
\newcommand{\bbbb}{\mathbb B}
\newcommand{\bbbs}{\mathbb S}

\newcommand{\bbbz}{\mathbb Z}
\newcommand{\N}{\mathbb N}

\newcommand{\R}{\mathbb R}
\newcommand{\overbar}[1]{\mkern 1.7mu\overline{\mkern-1.7mu#1\mkern-1.5mu}\mkern 1.5mu}

\newcommand{\eps}{\varepsilon}
\newcommand{\tf}{\tilde{f}}

\def\diam{\operatorname{diam}}
\def\dist{\operatorname{dist}}

\DeclareMathOperator*{\osc}{osc}
\newtheorem{theorem}{Theorem}
\newtheorem*{theorem*}{Theorem}
\newtheorem{lemma}[theorem]{Lemma}
\newtheorem{corollary}[theorem]{Corollary}
\newtheorem{proposition}[theorem]{Proposition}

\theoremstyle{definition}
\newtheorem{remark}[theorem]{Remark}
\newtheorem*{remark*}{Remark}
\newtheorem{example}[theorem]{Example}

\newcommand{\Sph}{\mathbb S}

\usepackage{setspace}
\title[Obstructions to continuity of Orlicz-Sobolev mappings]{Topological obstructions to continuity of Orlicz-Sobolev mappings of finite distortion}
\author[Goldstein]{Pawe\l{} Goldstein}
\address{Pawe\l{} Goldstein,\newline \indent Institute of Mathematics,\newline \indent Faculty of Mathematics, Informatics and Mechanics, \newline \indent University of Warsaw \newline \indent Banacha 2, 02-097 Warsaw, Poland}
\email{goldie@mimuw.edu.pl}
\thanks{P.G. was partially supported by NCN grant no 2012/05/E/ST1/03232.}
\author[Haj\l{}asz]{Piotr Haj\l{}asz}
\address{Piotr Haj\l{}asz,\newline \indent Department of Mathematics, University of Pittsburgh, \newline \indent 301 Thackeray Hall, Pittsburgh,
Pennsylvania 15260}
\email{hajlasz@pitt.edu}
\thanks{P.H. was supported by NSF grant.}


\begin{document}

\sloppy

\begin{abstract}
In the paper we investigate continuity of Orlicz-Sobolev mappings $W^{1,P}(M,N)$ of finite distortion between smooth Riemannian $n$-manifolds, $n\geq 2$, under the assumption that the Young function $P$ satisfies the so called divergence condition $\int_1^\infty P(t)/t^{n+1}\, dt=\infty$.
We prove that if the manifolds are oriented, $N$ is compact, and the universal cover of $N$ is
{\em not} a rational homology sphere, then such mappings are continuous. That includes mappings with $Df\in L^n$ and, more generally, mappings with $Df\in L^n\log^{-1}L$.
On the other hand, if the space $W^{1,P}$ is larger than $W^{1,n}$
(for example if $Df\in L^n\log^{-1}L$), and the universal cover of $N$ is homeomorphic to $\Sph^n$, $n\neq 4$, or is diffeomorphic to $\Sph^n$, $n=4$, then we construct an example of a mapping in
$W^{1,P}(M,N)$ that has finite distortion and is discontinuous.
This demonstrates a new {\em global-to-local phenomenon}: both finite distortion and continuity
are local properties, but a seemingly local fact that finite distortion implies continuity is a consequence of a global topological property of the target manifold $N$.
\end{abstract}

\maketitle
\section{Introduction}
Throughout  the paper we assume that $n\geq 2$.

Vodop'janov and Gol'd{\v{s}}te{\u\i}n \cite{VodoGold} proved that if a mapping $f:\Omega\to\bbbr^n$ of class $W^{1,n}$, defined on a domain $\Omega\subset\bbbr^n$,
has positive Jacobian, $J_f>0$, almost everywhere, then
$f$ is continuous (i.e. $f$ has a continuous representative).
In fact, Vodop'janov and Gol'd{\v{s}}te{\u\i}n proved a slightly stronger result, that $W^{1,n}$ mappings of finite distortion are continuous.

We say that a mapping $f\in W^{1,n}(\Omega,\bbbr^n)$, $\Omega\subset\bbbr^n$, has {\em finite distortion} if there is a function $K:\Omega\to [0,\infty)$ such that
$|Df(x)|^n\leq K(x)J_f(x)$ a.e. Taking $K(x)=|Df(x)|^n/J_f(x)$ we see that
$W^{1,n}$ mappings with almost everywhere positive Jacobian have finite distortion,
so continuity of mappings with finite distortion implies continuity of mappings with positive Jacobian.
The notion of finite distortion generalizes with no change to definition to mappings between oriented, smooth manifolds of the same dimension.

Recently, the result of Vodop'janov and Gol'd{\v{s}}te{\u\i}n has been extended to the case of mappings between manifolds \cite{GHP}.
\begin{theorem}
\label{T1}
Let $M$ and $N$ be smooth,  oriented, $n$-dimensional Riemannian manifolds without boundary and assume additionally that $N$ is compact.
If $f\in W^{1,n}(M,N)$ has finite distortion, then $f$ is continuous. In particular, if $f\in W^{1,n}(M,N)$ has positive Jacobian almost everywhere, then $f$
is continuous.
\end{theorem}

In the Euclidean setting the result of Vodop'janov and Gol'd{\v{s}}te{\u\i}n has been extended to Orlicz-Sobolev spaces that are larger than the Sobolev space $W^{1,n}$, see
\cite{iwaniecko}, \cite[Theorem~7.5.2]{IwaniecM}.

\begin{theorem}
\label{T2}
Let $P$ be a Young function satisfying the doubling condition \eqref{doubling cond}, the growth condition \eqref{growth 2} and the divergence condition \eqref{div condition}.
If $f\in W^{1,P}(\Omega,\bbbr^n)$, $\Omega\subset\bbbr^n$, has finite distortion, then $f$ is continuous.
\end{theorem}
For more information about Orlicz-Sobolev spaces, see Section~\ref{S2}.
If $P(t)=t^n$, then $W^{1,P}=W^{1,n}$, so the result of Vodop'janov and Gol'd{\v{s}}te{\u\i}n follows from Theorem~\ref{T2}. Perhaps the most important Young function, other than $P(t)=t^n$, satisfying the assumptions of Theorem~\ref{T2} is $P(t)=\frac{t^n}{\log(e+t)}$. The corresponding Orlicz-Sobolev space $W^{1,P}$ is larger than $W^{1,n}$. If $\Omega$ is bounded and has smooth boundary, then the space $W^{1,P}(\Omega)$ consists of functions such that
$$
|Df|\in L^n\log^{-1}L,
\quad
\text{that is}
\quad
\int_\Omega \frac{|Df|^n}{\log (e+|Df|)}<\infty.
$$
This is a much weaker condition than $L^n$ integrability of $|Df|$.

Orlicz-Sobolev spaces with Young functions satisfying the assumptions of Theorem~\ref{T2}
($L^n\log^{-1} L$ in particular) turn out to be critical in the study of regularity of distributional Jacobians, degree theory, properties of monotone mappings and the theory of Hardy-Orlicz spaces \cite{GH_RHS,grecoiss,HIMO,iwaniecko,IwaniecM,iwaniecs}.

In this paper we address the problem of extending Theorem~\ref{T2} to the case of Orlicz-Sobolev mappings
between manifolds.
Clearly, we want the Young function to satisfy the conditions described in Theorem~\ref{T2}, but since we
want the Orlicz-Sobolev space to be larger than the Sobolev space $W^{1,n}$, we impose one more growth
condition \eqref{growth 1a}. It turns out that in the case of manifolds, the answer to the question whether
$W^{1,P}(M,N)$ mappings of finite distortion (or even with almost everywhere positive Jacobian) are continuous depends on delicate topological properties of the target manifold
$N$. The main result of the paper reads as follows.
\begin{theorem}
\label{T3}
Let $M$ and $N$ be smooth,  oriented, $n$-dimensional Riemannian manifolds without boundary and assume additionally that $N$ is compact.
Assume that a Young function $P$ satisfies conditions \eqref{doubling cond}, \eqref{growth 1a},
\eqref{growth 2} and \eqref{div condition}.
\begin{itemize}
\item
If the universal cover of $N$ is not a rational homology sphere, then $W^{1,P}(M,N)$ mappings of finite distortion (in particular mappings with almost everywhere positive Jacobian) are continuous.
\item
If the universal
cover of $N$ is homeomorphic (when $n\neq 4$) or diffeomorphic (when $n=4$) to $\Sph^n$, then there are mappings in  $W^{1,P}(M,N)$ of finite distortion that are discontinuous.
If in addition $M=\bbbb^n$ is a Euclidean ball, one can construct a discontinuous mapping in $W^{1,P}(\bbbb^n,N)$ that has
almost everywhere positive Jacobian.
\end{itemize}
\end{theorem}
\begin{remark}
In fact in the first part (continuity) of Theorem~\ref{T3} we do not need the growth condition \eqref{growth 1a}, but a weaker condition \eqref{growth 1b}, ensuring that $W^{1,n}(M,N)\subset W^{1,P}(M,N)$ (see Theorems~\ref{T4} and~\ref{T5}). The growth condition \eqref{growth 1a} is only needed for the construction of a counterexample -- it guarantees that the space $W^{1,P}$ is strictly larger than $W^{1,n}$ and for $W^{1,n}$ mappings continuity is always
guaranteed by Theorem~\ref{T1}.
\end{remark}

We say that a compact $n$-manifold without boundary is a {\em rational homology sphere},
if it has the same deRham cohomology as the standard sphere $\Sph^n$. Rational homology spheres
were investigated in \cite{GH_RHS,grecoiss,HIMO} in the context of the degree theory of Orlicz-Sobolev
mappings.
Quasiregular mappings and mappings of finite distortion with values into rational homology spheres have also been studied in \cite{bonkh,onninenp}.
For more information about rational homology spheres, see Section~\ref{RHS}.

It follows from the Poincar\'e conjecture (when $n=3,4$) that
in dimensions $n=2,3,4$ simply connected rational homology spheres are homeomorphic (but for $n=4$ not necessarily diffeomorphic) to $\Sph^n$, so Theorem~\ref{T3} completely solves the
problem in dimensions $2$ and $3$. In dimension $4$ the situation is complicated by the possible existence of {\em exotic spheres}, see the discussion in Section \ref{sec disc}. However, in higher dimensions there is a gap, because
there are many examples of simply connected rational homology spheres that are not spheres, see Section~\ref{RHS}.

In geometry, the {\em local-to-global principle} means that: local properties of mappings imply their global properties.
However, Theorem~\ref{T3} shows a new, dual {\em global-to-local phenomenon}: both having  finite distortion (or positive Jacobian) and continuity
are local properties, but a seemingly local fact that finite distortion implies continuity is a consequence of a global topological property
of the target manifold $N$.

If the Young function $P$ satisfies the assumptions of Theorem~\ref{T3}, the derivative of a mapping $f\in W^{1,P}(M,N)$ does not necessarily belong to $L^n$
and hence there is no apparent reason why the Jacobian $J_f$ should be integrable. In fact, the discontinuous mappings discussed in the second part of Theorem~\ref{T3}
do not have integrable Jacobian. However, if we know that the Jacobian $J_f$ is integrable, then continuity of $f$ follows without any topological assumptions about $N$.
This is the second main result of the paper.
\begin{theorem}
\label{T4}
Let $M$ and $N$ be smooth,  oriented, $n$-dimensional Riemannian manifolds without boundary and assume additionally that $N$ is compact.
Assume that a Young function $P$ satisfies conditions \eqref{doubling cond}, \eqref{growth 1b}, \eqref{growth 2} and \eqref{div condition}.
If $f\in W^{1,P}(M,N)$ has finite distortion and $J_f\in L^1_{\rm loc}(M)$, then $f$ is continuous.
\end{theorem}
Now the first part (continuity) of Theorem~\ref{T3} follows from Theorem~\ref{T4} and the last main result.
\begin{theorem}
\label{T5}
Let $M$ and $N$ be smooth,  oriented, $n$-dimensional Riemannian manifolds without boundary and assume that $N$ is compact and the universal cover of $N$ is not a rational homology sphere.
If a Young function $P$ satisfies conditions \eqref{doubling cond}, \eqref{growth 1b}, \eqref{growth 2}, \eqref{div condition} and $f\in W^{1,P}(M,N)$
has non-negative Jacobian (in particular if $f$ has finite distortion), then the Jacobian $J_f$ is locally integrable, $J_f\in L^1_{\rm loc}(M)$.
\end{theorem}

\begin{remark}
Note that in Theorems~\ref{T4} and~\ref{T5} the growth condition \eqref{growth 1a} is not needed.
\end{remark}
\begin{remark}
Theorem~\ref{T5} is related to \cite[Theorem~6.6]{HIMO}, but the proof is very different and we do not know if the technique used in \cite{HIMO} can be adapted
to prove Theorem~\ref{T5}.
\end{remark}

The article is organized in the following way.
In Section~\ref{RHS} we discuss rational homology spheres.
In Section~\ref{S1} we recall basic results from the classical theory of Sobolev spaces. This section is followed by
Section~\ref{S2} devoted to definitions and facts from the theory of Orlicz-Sobolev spaces needed in the paper. As a byproduct of methods developed in that section, we provide a new proof of density of smooth mappings in the class of Orlicz-Sobolev mappings between manifolds, Corollary~\ref{cor:density}.
In Section~\ref{sec disc} we construct discontinuous maps in $W^{1,P}(M,N)$ that have finite distortion. This proves the second part of Theorem~\ref{T3}.
In Sections~\ref{Proof} and~\ref{Michal} we prove Theorems~\ref{T5} and~\ref{T4}, respectively. These results, along with results of Section~\ref{sec disc}, complete the proof of
Theorem~\ref{T3}.

In the final Section~\ref{Proof} we prove the continuity part of Theorem~\ref{T3} which, along with the results of Section~\ref{sec disc}, completes the proof of Theorem~\ref{T3}.

\subsection{Notation}
The notation used in the article is pretty standard. By $C$ we will denote a generic constant whose value may change even within a single string of estimates.
By writing, for example, $C=C(n,\alpha)$ we will mean that the constant $C$ depends on $n$ and $\alpha$ only.
The Lebesgue measure of a set $A$ (both in $\bbbr^n$ and on a manifold) will be denoted by $|A|$.
The volume of the unit Euclidean ball in $\bbbr^n$ will be denoted by $\omega_n$, so the volume of the unit sphere $\Sph^{n-1}$ is $n\omega_n$. The barred integral will denote the integral average
$$
\mvint_E f\, d\mu=\frac{1}{\mu(E)}\int_E f\, d\mu.
$$
The measure on a hypersurface in $\bbbr^n$ will be denoted by $d\sigma$.
The characteristic function of a set $E$ will be denoted by $\chi_E$. The closure of $E$ is denoted by $\overbar{E}$. By $\|\cdot\|_p$ we denote the $L^p$-norm with respect to the Lebesgue measure. Whenever we write about smooth functions or mappings, we mean $C^\infty$-smooth.

\subsection*{Acknowledgements}
The authors would like to thank Armin Schikorra for a helpful discussion about fractional Sobolev spaces.

\section{Rational homology spheres}
\label{RHS}

Let us recall that a compact $n$-dimensional Riemannian manifold $N$ without boundary
is a {\em rational homology sphere} if its deRham cohomology groups are the same as these of an $n$-dimensional sphere, i.e. $H^i_{dR}(N)=\bbbr$ for $i=0$ and $i=n$ and $H^i_{dR}(N)=0$ for all other values of $i$. The importance of this condition comes from the following lemma (\cite[Theorem 2.1]{GH_RHS}, see also \cite{Ruberman}).
\begin{lemma} \label{lem: RHS}
Let N be a smooth, compact, connected, oriented $n$-dimensional manifold without boundary. Then there is a smooth mapping
$f\colon \bbbs^n\to N$ of non-zero degree if and only if the universal cover of $N$ is a
rational homology sphere.
\end{lemma}
Thus if the universal cover of $N$ \emph{is not} a rational homology sphere, then every smooth mapping from an $n$-dimensional sphere to $N$ is of degree zero. The same holds for Lipschitz mappings from $\Sph^n$ to $N$, since every Lipschitz mapping is homotopic to a smooth one and the degree is a homotopy invariant.

Rational homology spheres include  spheres themselves,  integral homology spheres like the celebrated Poincar\'e sphere and more general Brieskorn
manifolds, and many others. The book \cite{Saveliev} provides numerous 3-dimensional examples, we refer the reader also to \cite[Section 2]{GH_RHS}. On the other hand, the following well known proposition holds:
\begin{proposition}
\label{P1}
If $N$ is an $n$-dimensional  rational homology sphere and
\begin{itemize}
\item[a)] $n=2$, then $N$ is diffeomorphic to $\Sph^2$,
\item[b)] $n=3$ and $N$ is simply connected, then $N$ is diffeomorphic to $\Sph^3$,
\item[c)] $n=4$, then $N$ is homeomorphic to $\Sph^4$.
\end{itemize}
\end{proposition}
\begin{proof}
Case a) follows from classification of closed surfaces (see e.g. \cite[Section 9.3]{Hirsch}): an orientable surface without boundary is uniquely (up to a diffeomorphism) determined by its genus, which, for a rational homology sphere, must be 0.

For a proof of c) see e.g. \cite[Proposition 2.6]{GH_RHS}.

The case $n=3$ is settled similarly as $n=4$, and we present a sketch of  arguments here.  The missing details and references can be found  in  \cite[Proposition 2.6]{GH_RHS}.

If $N$ is connected, orientable and compact, then $H_0(N,\bbbz)=\bbbz$ and $H_3(N,\bbbz)=\bbbz$. Since $N$ is also simply connected, $H_1(N,\bbbz)=0$, and the Universal Coefficients Theorem gives $H^1(N,\bbbz)=0$. Also,  Poincar\'e duality shows that $H_2(N,\bbbz)=H^1(N,\bbbz)=0$. Thus $N$ is an integral homology sphere; homology Whitehead's theorem yields that $N$ is a homotopy sphere. Finally, Perelman's theorem on Poincar\'e's Conjecture proves that $N$ is indeed diffeomorphic to a 3-dimensional sphere.
\end{proof}
\begin{corollary}
\label{P2}
Assume $N$ is a smooth, compact, connected, oriented $n$-dimensional manifold without boundary such that its universal cover  $\tilde{N}$ is a rational homology sphere. Then
\begin{itemize}
\item if $n=2$ then $N$ is diffeomorphic to $\Sph^2$,
\item if $n=3$, then $\tilde{N}$ is diffeomorphic to $\Sph^3$,
\item if $n=4$, then $N$ is homeomorphic to $\Sph^4$.
\end{itemize}
\end{corollary}
Whether $N$ is diffeomorphic to $\Sph^4$ when $n=4$, remains a long standing open problem.\\
However, in dimension 5 and higher, there are simply connected rational homology spheres that are not spheres, e.g. the Wu manifold $SU(3)/SO(3)$, \cite[Theorem 6.7]{MimuraToda} and \cite[Remark, p. 374]{Barden}. See \cite{Ruberman} and \cite[Lemma 1.1]{Barden} for more examples.

\section{Sobolev spaces}
\label{S1}

In this section we collect technical results from the theory of Sobolev spaces that are needed in the paper. All the results discussed here are well known except perhaps for Lemma~\ref{traces}, which is also known, but very difficult to find in the literature.

The Sobolev space $W^{1,p}(\Omega)$, $\Omega\subset\bbbr^n$,  consists of weakly differentiable functions such that
$\Vert f\Vert_{1,p}=\Vert f\Vert_p+\Vert Df\Vert_p<\infty$. Similarly we define the Sobolev space $W^{1,p}(M)$,
where $M$ is a Riemannian manifold.

Assume $M$ and $N$ are smooth, Riemannian manifolds without boundary, with $N$ compact. Assume also that $N$ is isometrically embedded in $\bbbr^{k}$ for some $k\in\N$. Then the class of Sobolev mappings
$W^{1,p}(M,N)$  is defined as
\begin{equation}
\label{SobolevMaps}
W^{1,p}(M,N)=\left\{f\in W^{1,p}_{\rm loc}(M,\bbbr^{k})~~\colon~~|Df|\in L^p(M) \text{ and }  f(x)\in N \text{ for a.e. }x\in M\right\}.
\end{equation}

If $M$ is compact, then $f\in L^p(M,\bbbr^k)$. However, we cannot require integrability of the mapping $f$ in the non-compact case (especially when the measure of $M$ is infinite):
if the measure of $M$ is infinite and if $N$ is embedded into $\bbbr^k$ in a way that the origin is at a positive distance to $N$, then no mapping
$f:M\to N\subset\bbbr^k$ is integrable with exponent $1\leq p<\infty$. Since our results are of local nature in $M$ (continuity and finite distortion are defined
locally), integrability of $f$ is not important to us.

\subsection{Pointwise inequalities}
For $f\in L^1_{\rm loc}(\bbbr^n)$ the Hardy-Littlewood maximal function is defined by
$$
\MM f(x)=\sup_{r>0}\mvint_{\bbbb^n(x,r)} |f(y)|\, dy.
$$
It is well known, \cite{stein}, that the maximal function is bounded in $L^p$ when $1<p<\infty$. This and Chebyshev's inequality imply that
if $f\in L^p(\bbbr^n)$, $1<p<\infty$, then
\begin{equation}
\label{weakLp}
t^p|\{x\in\bbbr^n:\, \MM f(x)>t\}|\to 0
\quad
\text{as $t\to\infty$.}
\end{equation}
If $f\in W^{1,1}_{\rm loc}$, then the following pointwise inequality is true, see for example \cite{acerbif,hajlasz1,lewis}.
\begin{equation}
\label{pointwise}
|f(x)-f(y)|\leq C(n)|x-y|(\MM |Df|(x)+\MM |Df|(y))
\quad
\text{for almost all $x,y\in\bbbr^n$.}
\end{equation}
A simple argument (see \cite[p.~97]{hajlasz1}) shows that if we choose the representative of $f$ defined by
$$
f(x):=\limsup_{r\to 0}\mvint_{\bbbb^n(x,r)} f(y)\, dy
\quad\text{for every $x\in\bbbr^n$,}
$$
then the inequality \eqref{pointwise} is true for {\em all} $x,y\in\bbbr^n$.

It easily follows
that if $f\in W^{1,1}(\bbbb^n)$, where $\bbbb^n\subset\bbbr^n$ is a ball of any radius, then the inequality \eqref{pointwise}
is still true for all $x,y\in\bbbb^n$, where we put $|Df|=0$ outside $\bbbb^n$, see \cite[Lemma~4~and~(7)]{hajlaszm}.

If $M$ is a compact Riemannian manifold with or without boundary, we can define the maximal function on $M$ with
averages over geodesic balls. Then the maximal function is bounded in $L^p(M)$, $1<p<\infty$, \eqref{weakLp} is true (with the same proof) and  $f\in W^{1,1}(M)$
satisfies the pointwise inequality
\begin{equation}
\label{pointwise2}
|f(x)-f(y)|\leq C(M)|x-y|(\MM |Df|(x)+\MM |Df|(y))
\quad
\text{for {\em all} $x,y\in M$,}
\end{equation}
with a suitable choice of a representative of $f$.
This inequality immediately gives the following well known
\begin{lemma}
\label{lemma lip}
Assume $f\in W^{1,1}(M)$, where $M$ is a compact Riemannian manifold with or without boundary.
Then $f$, restricted to the set $\{\MM |Df|\leq t\}$, is $Ct$-Lipschitz for some constant $C$ depending on $M$ only.
\end{lemma}

\subsection{Morrey's inequality}
The next result is a version of the classical Morrey's lemma in the case of Sobolev functions defined on a sphere.
\begin{lemma}
\label{morrey}
If $\Sph^{n-1}(r)$ is an $(n-1)$-dimensional Euclidean sphere of radius $r$ and
$f\in W^{1,\alpha}(\Sph^{n-1}(r))$ for some $\alpha >n-1$, then
$f$ has a $C^{0,1-\frac{n-1}{\alpha}}$-H\"older continuous representative
which satisfies
$$
\osc_{\Sph^{n-1}(r)} f=\sup_{x,y\in \Sph^{n-1}(r)}|f(x)-f(y)|\leq
C(n,\alpha)r\,\Big(\, \mvint_{\Sph^{n-1}(r)}|Df|^\alpha\, d\sigma\Big)^{1/\alpha}.
$$
\end{lemma}

\subsection{Traces}
\label{ssec:traces}
If $\Omega\subset\bbbr^n$ is a smooth bounded domain and $1<p<\infty$, we say that $u$ belongs to the fractional Sobolev space
$W^{1-\frac{1}{p},p}(\partial\Omega)$ if
$\Vert u\Vert_{1-\frac{1}{p},p}=\Vert u\Vert_{L^p(\partial\Omega)}+S_p(u)<\infty$, where
$$
S_p(u)=\left(\, \int_{\partial\Omega}\int_{\partial\Omega}\frac{|u(x)-u(y)|^p}{|x-y|^{n+p-2}}\, d\sigma(x)\, d\sigma(y)\right)^{1/p}\, .
$$
Gagliardo \cite{gagliardo} (see also \cite[Chapter~15]{leoni}) proved that the trace operator
${\rm Tr}\, :W^{1,p}(\Omega)\to W^{1-\frac{1}{p},p}(\partial\Omega)$ is bounded and there is an extension operator
${\rm Ext}:W^{1-\frac{1}{p},p}(\partial\Omega)\to W^{1,p}(\Omega)$ such that  ${\rm Tr}\circ{\rm Ext}={\rm id}$ on $W^{1-\frac{1}{p},p}(\partial\Omega)$.
In other words, the fractional Sobolev space $W^{1-\frac{1}{p},p}(\partial\Omega)$ provides a complete description of traces of $W^{1,p}(\Omega)$ functions.

In the article we will need the following known fact that was proven in \cite{BIN}.
A self contained and elementary (but difficult) proof can be found in \cite{leoni} (see Theorem~14.32, Remark~14.35 and Proposition~14.40). This result also follows from a sequence of results (as indicated below) in \cite{triebel}.
\begin{lemma}
\label{traces}
Let $\Omega\subset\bbbr^n$ be a bounded and smooth domain. If $n>2$, then
$W^{1,n-1}(\partial\Omega)\subset W^{1-\frac{1}{n},n}(\partial\Omega)$.
That is, there is a bounded extension operator ${\rm Ext}:W^{1,n-1}(\partial\Omega)\to W^{1,n}(\Omega)$.\\
If $n=2$, then for any $\alpha>1$ we have
$W^{1,\alpha}(\partial\Omega)\subset W^{\frac{1}{2},2}(\partial\Omega)$.
That is, there is a bounded extension operator
${\rm Ext}:W^{1,\alpha}(\partial\Omega)\to W^{1,2}(\Omega)$.
\end{lemma}
\begin{proof}
If $n>2$, then using the following results from \cite{triebel}:
Theorem~2.5.6, Theorem~2.7.1, Proposition~2.3.2.2(8), Theorem~2.5.7
and 2.5.7(9) (in that order) we obtain
the following relations for function spaces on $\bbbr^{n-1}$:
$$
W^{1,n-1}(\bbbr^{n-1})=
H^1_{n-1}=
F^1_{n-1,2}\subset
F^{1-\frac{1}{n}}_{n,n}=
B^{1-\frac{1}{n}}_{n,n}=
\Lambda^{1-\frac{1}{n}}_{n,n}=
W^{1-\frac{1}{n},n}(\bbbr^{n-1}).
$$
The above identification of Sobolev spaces and Triebel-Lizorkin spaces fails for $n=2$,
and in that case we need to argue in a slightly different way.
Note that we can assume that $1<\alpha<2$. Then using the results from
\cite{triebel}: Theorem~2.5.6, Proposition~2.3.2.2(8), Theorem~2.7.1, Proposition~2.3.2.2(9), Theorem~2.5.7(5) and 2.5.7(9)
(in that order) we obtain
$$
W^{1,\alpha}(\bbbr)=H^1_\alpha=F^1_{\alpha,2}\subset F^{\frac{1}{\alpha}}_{\alpha,2}
\subset F^{\frac{1}{2}}_{2,2}=B^{\frac{1}{2}}_{2,2}=\Lambda^{\frac{1}{2}}_{2,2}=W^{\frac{1}{2},2}(\bbbr).
$$\end{proof}

\begin{remark}
In fact, if $n=2$, the space $W^{1,1}(\partial\Omega)$ does not embed into
$W^{\frac{1}{2},2}(\partial\Omega)$,
see \cite{BIN}, \cite[Exercise~14.36]{leoni} and \cite[Proposition~4]{SW}.
\end{remark}
\begin{remark}
The reasoning in the case $n=2$ is slightly different than in the case $n>2$ and it has to be different as is explained below.
If we apply the argument from the case $n>2$ to $W^{1,\alpha}(\bbbr)$, $1<\alpha<2$,
then we obtain
$$
W^{1,\alpha}=H^1_\alpha=F^{1}_{\alpha,2}\subset F^{\frac{1}{2}}_{\frac{2\alpha}{2-\alpha},\frac{2\alpha}{2-\alpha}}=
B^{\frac{1}{2}}_{\frac{2\alpha}{2-\alpha},\frac{2\alpha}{2-\alpha}} =
\Lambda^{\frac{1}{2}}_{\frac{2\alpha}{2-\alpha},\frac{2\alpha}{2-\alpha}} =
W^{\frac{1}{2},\frac{2\alpha}{2-\alpha}}(\bbbr).
$$
Since $\frac{2\alpha}{2-\alpha}>2$ it seems that we get what we wanted
$$
W^{1,\alpha}(\partial\Omega)\subset W^{\frac{1}{2},\frac{2\alpha}{2-\alpha}}(\partial\Omega)\subset W^{\frac{1}{2},2}(\partial\Omega).
$$
However, very surprisingly, the last inclusion is {\bf false} as was shown in \cite{mironescus}.
\end{remark}
\begin{remark}
\label{szesnascie}
The extension operator ${\rm Ext}:W^{1-\frac{1}{p},p}(\partial\Omega)\to W^{1,p}(\Omega)$ can be defined by an explicit integral formula
\cite[Theorem~15.21]{leoni} from which it follows that the extension is smooth in $\Omega$
and is continuous up to the boundary if the function that we extend is continuous on $\partial\Omega$.
\end{remark}

\section{Orlicz-Sobolev spaces}
\label{S2}
In this section we briefly describe basic properties of Orlicz-Sobolev spaces that are used in the paper.
Since we do not need any delicate results from the theory of Orlicz-Sobolev spaces, we will try to keep the
definitions as simple as possible.
For more information about Orlicz and Orlicz-Sobolev spaces, see e.g. \cite{AF,IwaniecM,KR,RR}. Our approach is closely related to that in
\cite[Section~4]{HIMO}.

We assume that  $P:[0,\infty) \to [0,\infty)$ is convex, strictly increasing, with $P(0)=0$.
A function satisfying these conditions is called a \emph{Young function}.
We will always assume that $P$ satisfies the {\em doubling condition}:
\begin{itemize}
\item\textbf{doubling condition} or \textbf{$\Delta_2$-condition}
\begin{equation}
\label{doubling cond}
\text{there exists }K>0\text{ such that }P(2t)\leq K\,P(t)\text{ for all }t\geq 0.
\end{equation}
\end{itemize}

We also consider other conditions:

\begin{itemize}
\item\textbf{growth conditions}
\begin{subequations}
\begin{equation}
\frac{P(t)}{t^n}\rightarrow 0\quad\text{ as }t\to\infty,\label{growth 1a}
\end{equation}
\begin{equation}
P(t)\leq C t^n \text{  for some }C>0 \text{ and all }t\geq 1.\label{growth 1b}
\end{equation}
\end{subequations}
\smallskip
\begin{equation}
\text{the function } t^{-\alpha} P(t) \text{ is non-decreasing for some } n>\alpha>n-1,\label{growth 2}
\end{equation}
\item\textbf{divergence condition}
\begin{equation}
\label{div condition}
\int_1^\infty \frac{P(t)}{t^{n+1}}\, dt=\infty.
\end{equation}
\end{itemize}

Let $(X,\mu)$ be a measure space. If a Young function $P$ satisfies the doubling condition
\eqref{doubling cond}, then the {\em Orlicz space} $L^P(X)$ is defined as the class of all measurable functions $f$
such that $\int_X P(|f|)\, d\mu<\infty$. Clearly, $L^P\subset L^1_{\rm loc}$.

The doubling condition implies that $L^P$ is a linear space and actually it is a Banach space with respect to the
{\em Luxemburg norm}
$$
\Vert f\Vert_P=\inf\left\{ k>0:\, \int_X P(|f|/k)\, d\mu\leq 1\right\}\, .
$$
We say that a sequence $(f_k)$ of functions in $L^P(X)$ converges to $f$ {\em in mean}, if
$$
\lim_{k\to\infty} \int_X P(|f_k-f|)\, d\mu=0.
$$
It is well known and easy to show that under the doubling condition convergence in mean is equivalent to convergence in the Luxemburg norm.

Given a Young function $P$ satisfying \eqref{doubling cond}, we define the {\em Orlicz-Sobolev space} $W^{1,P}(\Omega)$ on a domain $\Omega\subset\bbbr^n$
as the class of functions in $W^{1,1}_{\rm loc}(\Omega)$ such that
\begin{equation}
\label{OSNorm}
\Vert f\Vert_{1,P}=\Vert f\Vert_P+\Vert Df\Vert_P<\infty.
\end{equation}
$W^{1,P}(\Omega)$ is a Banach space and smooth functions, $C^\infty(\Omega)$, are dense in it. Similarly as in the case of $L^P$, it is not difficult to show that
$f_k\to f$ in $W^{1,P}$ if and only if we have convergence in mean, i.e.
$$
\int_\Omega P(|f-f_k|)+P(|Df-Df_k|)\, dx\to 0
\quad
\text{as $k\to\infty$.}
$$
The definition of the Orlicz-Sobolev space can be easily extended
to the case of functions on a Riemannian manifold, $W^{1,P}(M)$.
Then Orlicz-Sobolev mappings between manifolds $W^{1,P}(M,N)$ are defined by analogy to \eqref{SobolevMaps}.

The next lemma was proved in \cite{HIMO}: part a) is obvious.
For part b) see \cite[Proposition~4.7]{HIMO}.
Part c) is an easy consequence of part b) and for part d) see \cite[Proposition~4.1]{HIMO} and Lemma~\ref{morrey}.
For the sake of completeness we will include a short and self-contained proof.

\begin{lemma}
\label{lem:Eucl}
Assume $f\in W^{1,P}(M)$, where $M$ is a smooth, compact, Riemannian \mbox{$n$-manifold} with or without boundary.
Assume moreover that $P$ satisfies conditions \eqref{doubling cond}, \eqref{growth 2} and \eqref{div condition}. Then
\begin{itemize}
\item[a)] $\displaystyle f\in W^{1,\alpha}(M)$,
\item[b)] $\displaystyle \liminf_{t\to\infty} t^{n-\alpha} \int_{\{|Df|>t\}}|Df|^{\alpha}=0$,
\item[c)] $\displaystyle \liminf_{t\to\infty} t^n \left|\left\{x\in M~:~\MM |Df|>t\right\}\right|=0$,
\item[d)] if $x\in M\setminus\partial M$ and $R>0$, then for any $\eps>0$, the set of radii $r\in (0,R)$ such that
$\osc_{\Sph^{n-1}(x,r)} f=\sup_{y,z\in \Sph^{n-1}(x,r)} |f(y)-f(z)|<\eps$ has positive linear measure.
\end{itemize}
\end{lemma}
\begin{remark}\label{rem morrey}
Since for almost all $r\in (0,R)$ we have $f\in W^{1,\alpha}(\Sph^{n-1}(x,r))$, where $\alpha>n-1$, the function $f$ restricted to
$\Sph^{n-1}(x,r)$ has a H\"older continuous representative and the oscillation in d) is defined for that representative.
\end{remark}
\begin{proof}
Note that \eqref{growth 2} implies that
$P(t)\geq P(1)t^{\alpha}$ for $t\geq 1$, which immediately gives a).

Let $\alpha$ be as in \eqref{growth 2} and set $\Psi(t)=t^{\alpha-n}(t^{-\alpha}P(t))'$.
Then, integrating by parts, we get that
$$
\int_{1}^k\Psi(t)\, dt=\int_{1}^k t^{\alpha-n}(t^{-\alpha}P(t))'\, dt=
t^{-n}P(t)\Big|_1^k+(n-\alpha)\int_{1}^k\frac{P(t)}{t^{n+1}}\, dt
$$
and letting $k\to\infty$ yields
\begin{equation}
\label{eq1}
\int_1^\infty \Psi(t)\, dt = +\infty.
\end{equation}
On the other hand, for all $T>1$,
\begin{equation*}
\begin{split}
&\left(\int_{1}^T \Psi(t)\, dt\right)\inf_{t>1}\left(t^{n-\alpha}\int_{\{|Df|>t\}}|Df|^\alpha\, dx\right)
\leq \int_{1}^T\Psi(t)t^{n-\alpha}\int_{\{|Df|>t\}}|Df|^\alpha\, dx\,dt\\
&\stackrel{*}{\leq}\int_{\{|Df|>1\}}|Df|^\alpha \left(\int_{1}^{|Df|}\Psi(t) t^{n-\alpha}\,dt\right)\, dx
=\int_{\{|Df|>1\}}|Df|^\alpha \left(\int_{1}^{|Df|}(t^{-\alpha} P(t))'\, dt\right)\, dx\\
&= \int_{\{|Df|>1\}}|Df|^\alpha\left(|Df|^{-\alpha} P(|Df|)-P(1)\right)\, dx\leq
\int_M P(|Df|)\, dx<\infty,
\end{split}
\end{equation*}
where the inequality $\stackrel{*}{\leq}$ follows from the obvious implication
$$
\left(1\leq t\leq T,\, |Df|>t\right) \Rightarrow \left( |Df|>1,\, 1\leq t<|Df|\right).
$$
Letting $T\to\infty$ and using \eqref{eq1} we see  that
$\inf_{t>1}\left(t^{n-\alpha}\int_{\{|Df|>t\}}|Df|^\alpha\right)=0$, which, in turn, implies b).

Let $h=|Df|\chi_{\{|Df|>t/2\}}$. Then $|Df|\leq h+t/2$, so
$\{\MM |Df|>t\}\subset\{\MM h>t/2\}$ and Chebyshev's inequality along with boundedness of the maximal function on
$L^\alpha(M)$ yield
\begin{equation*}
\begin{split}
t^n |\{\MM |Df|>t\}|&\leq t^n |\{\MM h>t/2\}|
\leq t^n \left(\frac{2}{t}\right)^\alpha \Vert \MM h\Vert_\alpha^\alpha\\
&\leq C(M,\alpha)t^{n-\alpha} \Vert h\Vert_\alpha^\alpha
= C t^{n-\alpha}\int_{\{|Df|>t/2\}}|Df|^\alpha.
\end{split}
\end{equation*}
Combining this estimate with b) immediately yields c).

It remains to prove d).

There is $0<R'\leq R$ such that the exponential map $\exp_x:T_x M\to M$ maps Euclidean balls
$\bbbb^n(0,r)\subset T_x M$ onto Riemannian balls $\bbbb^n(x,r)\subset M$, in a diffeomorphic way, for all $0<r\leq R'$
 (see \cite{jost}).
In particular, it maps Euclidean spheres centered at $0$ onto spheres in $M$ centered at $x$.
Thus we can assume that
$\bbbb^n(x,R')=\bbbb^n(0,R')$ is the Euclidean ball.

According to Lemma~\ref{morrey} it suffices to show that the set of $r\in (0,R')$ such that
$$
r\, \Bigg(\, \mvint_{\Sph^{n-1}(0,r)}|Df|^\alpha\, d\sigma\Bigg)^{1/\alpha}<\eps
$$
has positive linear measure. Suppose to the contrary, that there is $\eps>0$ such that
$$
r\, \Bigg(\, \mvint_{\Sph^{n-1}(0,r)}|Df|^\alpha\, d\sigma\Bigg)^{1/\alpha}\geq \eps
\quad
\text{for almost all $r\in (0,R')$.}
$$
Integration in spherical coordinates gives that for every $0<\rho<R'$
$$
\int_{\bbbb^n(0,\rho)} |Df|^\alpha\, dx\geq \frac{n\omega_n\eps^\alpha\rho^{n-\alpha}}{n-\alpha}\, .
$$
For $t>0$, $|Df|^\alpha\leq |Df|^\alpha\chi_{\{|Df|>t\}}+t^\alpha$, so
$$
\frac{n\omega_n\eps^\alpha\rho^{n-\alpha}}{n-\alpha}\leq
\int_{\bbbb^n(0,\rho)} |Df|^\alpha\, dx\leq\int_{\{|Df|>t\}}|Df|^\alpha+\omega_n\rho^nt^\alpha\, .
$$
Let $t>0$ be such that
$$
\frac{n\omega_n\eps^\alpha\rho^{n-\alpha}}{n-\alpha} = 2\omega_n\rho^nt^\alpha\, ,
\quad
\text{that is}
\quad
t=\left(\frac{n}{2(n-\alpha)}\right)^{1/\alpha}\frac{\eps}{\rho}\, .
$$
Then
$$
\int_{\{|Df|>t\}} |Df|^\alpha\, dx\geq \omega_n\rho^nt^\alpha =
\omega_n\eps^n \left(\frac{n}{2(n-\alpha)}\right)^{n/\alpha} t^{\alpha-n}\, ,
$$
$$\text{thus}\qquad  C(n,\alpha)\eps^n\leq t^{n-\alpha} \int_{\{|Df|>t\}} |Df|^\alpha\, dx.
$$
When $\rho\to 0$, we have $t\to\infty$, and hence
$$
\liminf_{t\to\infty} t^{n-\alpha}\int_{\{|Df|>t\}} |Df|^\alpha\, dx\geq
C(n,\alpha)\eps^n,
$$
which contradicts b).
\end{proof}

\begin{lemma}
\label{lemma approx}
Assume $f\in W^{1,P}(M,N)$, where $M$ and $N$ are compact Riemannian manifolds and $\dim M=n$.
$M$ may have boundary, but $N$ has no boundary.
Assume moreover that $P$ satisfies conditions \eqref{doubling cond},  \eqref{growth 2} and \eqref{div condition}.
Then there exists a constant $C$ depending on $M$ and $N$ only, a sequence $t_i\to\infty$, and continuous maps $f_i\in C(M,N)$ such that
\begin{itemize}
\item[a)] $f_i$ is $Ct_i$-Lipschitz,
\item[b)] $f_i$ coincides with $f$ on the set $\{\MM |Df| \leq t_i\}$,
\item[c)] $\lim_{i\to\infty} t_i^n \left|\left\{x\in M~:~\MM |Df|>t_i\right\}\right|=0$.
\end{itemize}
\end{lemma}
\begin{proof}
Recall that, by Lemma~\ref{lem:Eucl} c),
$\liminf_{t\to\infty} t^n \left|\left\{x\in M~:~\MM |Df|>t\right\}\right|=0$. Thus, let $(t_i)$ denote such an increasing sequence that $\lim_{i\to\infty}t_i=+\infty$ and
\begin{equation}\label{lab 1}
\lim_{i\to\infty} t_i^n \left|\left\{x\in M~:~\MM |Df|>t_i\right\}\right|=0.
\end{equation}
By Lemma \ref{lemma lip},
$$
f\big|_{\{\MM |Df|\leq t_i\}}\colon \{\MM |Df|\leq t_i\}\to N\subset\bbbr^k
\qquad
\text{is $C t_i$-Lipschitz.}
$$
Let $F_i:M\to\R^k$ denote its $C t_i$-Lipschitz extension to the whole $M$. Then $F_i$ coincides with $f$ on $\{\MM |Df|\leq t_i\}$ and is $C t_i$-Lipschitz; however, it does not necessarily map $M$ to~$N$.

Let $d_i$ be the radius of the largest Riemannian ball contained in the
open set $\{\MM |Df|>t_i\}$. Then, for any $x\in \{\MM |Df|> t_i\}$ we can find $y\in \{\MM (|Df|)\leq t_i\}$ such that the Riemannian distance $d(x,y)$ is not more than $d_i$. Then,
$$
\dist(F_i(x),N)\leq |F_i(x)-f(y)|=|F_i(x)-F_i(y)|\leq C t_i |x-y|\leq C t_i d_i.
$$

Note also that, by \eqref{lab 1}, we have $\lim_{i\to\infty}t_i^n d_i^n=0$, thus $\lim_{i\to\infty}t_i d_i=0$, which implies that
$\dist(F_i(x),N)\leq  C t_i d_i \xrightarrow{i\to\infty} 0$. This convergence is uniform in $x$, thus the image of the mapping $F_i$ lies in a tubular neighborhood of $N$.
As $N$ is compact, for a sufficiently small tubular neighborhood of $N$ there is a well defined, smooth and Lipschitz nearest point projection $\pi$ onto $N$; taking $f_i=\pi\circ F_i$ gives the desired sequence. Claims a) and b) are then obvious.
\end{proof}

As a corollary we obtain a new and much shorter proof of Theorem~5.2 in \cite{HIMO}. For some generalizations of this result, see \cite{ccianchi}.

\begin{corollary}
\label{cor:density}
Assume $f\in W^{1,P}(M,N)$, where $M$ and $N$ are compact Riemannian manifolds and $\dim M=n$.
$M$ may have boundary, but $N$ has no boundary.
Assume moreover that $P$ satisfies conditions \eqref{doubling cond}, \eqref{growth 1b}, \eqref{growth 2} and \eqref{div condition}. Then $C^{\infty}(M,N)$ is dense in $W^{1,P}(M,N)$.
\end{corollary}
\begin{proof}
Assume $f\in W^{1,P}(M,N)$. By Lemma \ref{lemma approx} there is a sequence $t_i\to\infty$ and $Ct_i$\nobreak\mbox{-}\nobreak{}Lipschitz functions $f_i\in C(M,N)$, $i=1,2,\ldots$, that coincide with $f$ on the sets $\{\MM |Df|\leq t_i\}$, respectively.

We shall prove that $f_i$ converge to $f$ in $W^{1,P}(M,N)$ {\em in the mean}.

Let $A_N=\max\{|x|~:~x\in N\}$ be the maximum distance of a point in $N$ to the origin. Naturally, $\|f_i\|_\infty\leq A_N$. Also, since $f_i$ is $Ct_i$-Lipschitz, $|Df_i|\leq Ct_i$ a.e.

Convexity of $P$ and the doubling condition \eqref{doubling cond} imply immediately that for $a,b\geq 0$ we have $P(a+b)\leq \frac{K}{2}(P(a)+P(b))$. Also, by \eqref{growth 1b}, we have $P(t)\leq C t^n$ for large $t$.
\begin{equation}
\begin{split}
\int_M &\left(P(|f-f_i|)+P(|Df-Df_i|)\right)=
\int_{\{\MM |Df|>t_i\}}\left(P(|f-f_i|)+P(|Df-Df_i|)\right)\\
&\qquad\leq \frac{K}{2}\int_{\{\MM |Df|>t_i\}}\left(P(|f|)+P(A_N)+P(|Df|)+P(Ct_i)\right)\\
&\qquad \leq \frac{K}{2} \int_{\{\MM |Df|>t_i\}} \left(P(|f|)+P(|Df|)\right)+
\frac{K}{2}|\{\MM |Df|>t_i\}|(P(A_N)+Ct^n_i)\\
&\qquad=I_i+J_i.
\end{split}
\end{equation}
The term $I_i$ tends to 0 as $i\to\infty$, because the function $P(|f|)+P(|Df|)$ is integrable, and the measure of $\{\MM |Df|>t_i\}$ goes to zero.
Likewise, $J_i\to 0$ as $i\to \infty$, by Lemma \ref{lemma approx} c).

This proves the density of Lipschitz maps in $W^{1,P}(M,N)$. To pass to smooth maps, it suffices to show that we can approximate a Lipschitz map $g \in C(M,N)$ by a smooth one {\em in the mean}.

Since $N$ is isometrically embedded in $\R^k$,
there is a sequence $\tilde{g}_i\in C^{\infty}(M,\R^k)$ of smooth maps uniformly approximating $g$, obtained by a standard convolution with a mollifier. One immediately checks that $\tilde{g}_i$ have derivatives uniformly bounded by the Lipschitz constant $L$ of $g$. Then we compose the $\tilde{g}_i$ with the nearest point projection $\pi$ onto $N$ (by uniform convergence of $\tilde{g_i}$, for large $i$ the images of $\tilde{g_i}$ lie arbitrarily close to $N$), getting $g_i=\pi\circ\tilde{g}_i\in C^{\infty}(M,N)$. Then $|Dg_i|$ are uniformly bounded by $CL$ for some $C$, $g_i$ converge uniformly to $g$ and $Dg_i$ converge to $Dg$ a.e. in $M$. The Dominated Convergence Theorem yields that $g_i$ converge to $g$ in $W^{1,P}(M,N)$ in the mean.
\end{proof}

\section{Discontinuous mappings in $W^{1,P}(M,N)$}\label{sec disc}
\begin{proposition}
\label{prop disc}
Assume $M$, $N$ are smooth, oriented $n$-dimensional Riemannian manifolds without boundary, $N$ is compact and the universal cover $\tilde{N}$ of $N$ is either diffeomorphic (if $n=4$) or homeomorphic (if $n\neq 4$) to $\Sph^n$.
Assume moreover that $P$ satisfies conditions
\eqref{doubling cond}, \eqref{growth 1a}, \eqref{growth 2} and \eqref{div condition}.
Then there exists a discontinuous mapping $G\in W^{1,P}(M,N)$ of finite distortion. If in addition $M=\bbbb^n$ is a Euclidean ball, then we can construct a discontinuous mapping $G\in W^{1,P}(\bbbb^n,N)$ that has almost everywhere positive Jacobian.
\end{proposition}

First, let us explain the difference between the dimension 4 and all other dimensions (recall that, throughout the paper, $n\geq 2$). In dimensions $n=2,3,5,6$, if a smooth $n$-manifold $\tilde{N}$ is homeomorphic to $\Sph^n$, then it is diffeomorphic to $\Sph^n$. When $n=2$ it follows from classification of surfaces. When $n=3$, it follows from
Hauptvermutung for $3$-manifolds \cite[Theorems~3~and~4]{moise} and from the fact that a combinatorial $3$-manifold has a unique smoothing \cite[Theorem~4.2]{hirschm}. For the cases $n=5,6$, see \cite[Remark~on~p.~505]{kervairem}.
In dimensions $n\geq 7$ there exist {\em exotic spheres}: smooth manifolds $\tilde{N}$ homeomorphic, but not diffeomorphic to $\Sph^n$, see \cite[Remark~on~p.~505]{kervairem}. The first example of an exotic sphere was discovered by Milnor \cite{milnor} in dimension $n=7$.

Assume $\tilde{N}$ is an exotic sphere and $n=\dim \tilde{N}\geq 7$. Even though there is no diffeomorphism between $\Sph^n$ and $\tilde{N}$, one can find a bi-Lipschitz homeomorphism between these spaces, see \cite[Theorem 2]{Sullivan} and \cite[Corollary 4.5]{VaisalaTukia}. We can even assume it is a diffeomorphism outside a single point, see \cite[Corollary 1.15]{KondoTanaka}.

The situation in dimension $4$ is more complicated: it is not known whether there exist any exotic sphere of that dimension, or, if there exists one, whether it is
bi-Lipschitz-homeomorphic to the standard $\Sph^4$. Therefore we need the assumption that if $n=4$,  $\tilde{N}$ is diffeomorphic to $\Sph^4$, as a safeguard
against the possible existence of exotic spheres in dimension 4. We could weaken it, asking that $\tilde{N}$ is bi-Lipschitz equivalent to $\Sph^4$.

The main ingredient of the proof of Proposition \ref{prop disc} is an example of a discontinuous mapping $F\in W^{1,P}(\Sph^n,\Sph^n)$ of positive Jacobian a.e. This example has appeared first in \cite[Section 3.4]{HIMO}, and the details of the construction are carefully presented there. In \cite[Section 4]{GHP} a slightly simplified construction in the case $P(t)=t^n/\log(e+t)$ is given. We refer the reader to these two sources, giving here only a sketch of the construction.
\begin{example}\label{example}
Let
$
S^\beta_\alpha=\{(z\sin\theta,\cos \theta)~~:~~z\in\bbbs^{n-1},\,\alpha\leq\theta\leq\beta\}\subset\bbbs^n
$
be the spherical slice bounded by latitude spheres $\theta=\alpha$ and $\theta=\beta$ (with $\theta=0$ denoting the north and $\theta=\pi$
the south pole of $\bbbs^n$).

We choose a sequence of latitude angles $\theta_i$, $i=1,2,\ldots$ such that
$$
\theta_0 =\pi>\frac{\pi}{2}\geq 2\theta_1>\theta_1\geq 2\theta_2>\theta_2\geq 2\theta_3>\cdots>0.
$$
The mapping $F:\Sph^n\to\Sph^n$  is identity on the spherical slices $S_{2\theta_{k+1}}^{\theta_k}$ and on the southern cap $S^{\pi}_{2\theta_1}$. The slice  $S^{2\theta_k}_{\theta_k}$, for $k=1,2,\ldots$,  is stretched (linearly in the latitude angle) in such a way that it covers the whole sphere twice, keeping the latitude circle $\theta=\theta_k$ fixed, mapping $\theta=3\pi\theta_k/(2\pi+\theta_k)$ to the south pole, $\theta=4\pi\theta_k/(2\pi+\theta_k)$ to the north pole and finally the latitude circle $\theta=2\theta_k$ back to itself (Figure \ref{fig:1}). We can do it in an orientation-preserving way, so the Jacobian determinant of $F$ is positive inside the slice.
%------------------
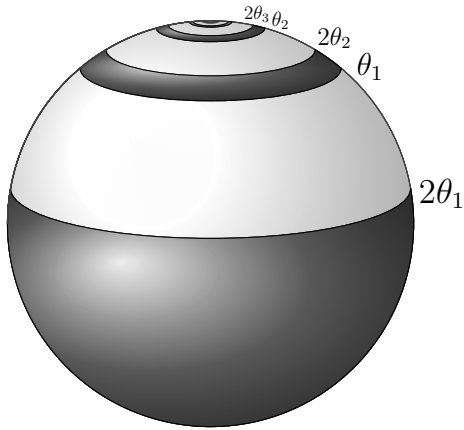
\begin{figure}
\begin{minipage}{0.3\textwidth}
\begin{tikzpicture}[scale=0.9]
%\draw[help lines] (0,-5) grid (15,5);

\filldraw[shading=ball, ball color=white,opacity=0.8] (3,0) circle (3);
\filldraw[fill=white,opacity=0.5] (3,0) circle (3);
\draw[color=darkgray] (3,0) circle (3);
\draw[color=darkgray] (0,0) arc [radius=3, start angle=180, end angle=170,opacity=0] arc [x radius={3*sin(80)}, y radius={0.75*sin(80)}, start angle=-180, end angle=0];
\draw[color=darkgray] (3,0) circle (3);
\draw[color=darkgray] (0,0) arc [radius=3, start angle=180, end angle=170,opacity=0] arc [x radius={3*sin(80)}, y radius={0.75*sin(80)}, start angle=-180, end angle=0];

\filldraw[shading=ball,ball color=gray] ({3+3*cos(10)},{3*sin(10)})  arc [x radius={3*sin(80)}, y radius={0.75*sin(80)}, start angle=0, end angle=-180]  arc [radius=3, start angle=170, end angle=370];

%\draw[thick,color=lightgray] ({3+3*cos(220)},{3*sin(220)})  arc [x radius={3*sin(40)+0.4}, y radius={0.8*sin(40)}, start angle=-170, end angle=-10];

\draw[color=darkgray] (0,0) arc [radius=3, start angle=180, end angle=130,opacity=0] arc [x radius={3*sin(40)}, y radius={0.75*sin(40)}, start angle=-180, end angle=0];
\draw[color=darkgray] (0,0) arc [radius=3, start angle=180, end angle=121,opacity=0] arc [x radius={3*sin(31)}, y radius={0.75*sin(31)}, start angle=-180, end angle=0];

\filldraw[shading=ball,ball color=gray] ({3+3*cos(130)},{3*sin(130)}) arc [x radius={3*sin(40)}, y radius={0.75*sin(40)}, start angle=-180, end angle=0] arc [radius=3, start angle=50, end angle=59] arc [x radius={3*sin(31)}, y radius={0.75*sin(31)}, start angle=0, end angle=-180] arc [radius=3, start angle=121, end angle=130];

\draw[color=darkgray] (0,0) arc [radius=3, start angle=180, end angle=105.5,opacity=0] arc [x radius={3*sin(15.5)}, y radius={0.75*sin(15.5)}, start angle=-180, end angle=0];
\draw[color=darkgray] (0,0) arc [radius=3, start angle=180, end angle=102,opacity=0] arc [x radius={3*sin(12)}, y radius={0.75*sin(12)}, start angle=-180, end angle=0];

\filldraw[shading=ball,ball color=gray] ({3+3*cos(105.5)},{3*sin(105.5)}) arc [x radius={3*sin(15.5)}, y radius={0.75*sin(15.5)}, start angle=-180, end angle=0] arc [radius=3, start angle=74.5, end angle=78] arc [x radius={3*sin(12)}, y radius={0.75*sin(12)}, start angle=0, end angle=-180] arc [radius=3, start angle=102, end angle=105];

\draw[color=darkgray] (0,0) arc [radius=3, start angle=180, end angle=96,opacity=0] arc [x radius={3*sin(6)}, y radius={0.75*sin(6)}, start angle=-180, end angle=0];
\draw[color=darkgray] (0,0) arc [radius=3, start angle=180, end angle=95,opacity=0] arc [x radius={3*sin(5)}, y radius={0.75*sin(5)}, start angle=-180, end angle=0];

\filldraw[fill=gray] ({3+3*cos(96)},{3*sin(96)}) arc [x radius={3*sin(6)}, y radius={0.75*sin(6)}, start angle=-180, end angle=0] arc [radius=3, start angle=84, end angle=85] arc [x radius={3*sin(5)}, y radius={0.75*sin(5)}, start angle=0, end angle=-180] arc [radius=3, start angle=95, end angle=96];

\draw[color=darkgray] (0,0) arc [radius=3, start angle=180, end angle=92.5,opacity=0] arc [x radius={3*sin(2.5)}, y radius={0.75*sin(2.5)}, start angle=-180, end angle=0];
\draw[color=darkgray] (0,0) arc [radius=3, start angle=180, end angle=92.1,opacity=0] arc [x radius={3*sin(2.1)}, y radius={0.75*sin(2.1)}, start angle=-180, end angle=0];
\draw[color=darkgray] (0,0) arc [radius=3, start angle=180, end angle=91.7,opacity=0] arc [x radius={3*sin(1.7)}, y radius={0.75*sin(1.7)}, start angle=-180, end angle=0];
\draw[color=darkgray] (0,0) arc [radius=3, start angle=180, end angle=91.4,opacity=0] arc [x radius={3*sin(1.4)}, y radius={0.75*sin(1.4)}, start angle=-180, end angle=0];

%\node[below] at (3,-3) {$\theta_0$};
\node[right, scale=1.1] at (5.9,0.45) {$2\theta_1$};
\node[right] at (5,2.32) {$\theta_1$};
\node[right,scale=0.8] at (4.45,2.75) {$2\theta_2$};
\node[right,scale=0.65] at (3.8,3) {$\theta_2$};
\node[right,scale=0.6] at (3.4,3.1) {$2\theta_3$};
\end{tikzpicture}
\end{minipage}
\hfill
\begin{minipage}{0.65\textwidth}
\caption{Example \ref{example}: Dark shaded bands are preserved by $F$, while light shaded ones are stretched onto the whole sphere so that they cover the whole sphere twice, in an orientation preserving way.}
\end{minipage}
\label{fig:1}
\end{figure}
%----------------------------------------

Obviously, $F$ is discontinuous in the north pole $\mathfrak{n}$
and, since the image of each slice  $S^{2\theta_k}_{\theta_k}$ has measure $2|\Sph^n|$, the Jacobian of $F$ is not integrable.

As shown in \cite{HIMO}, if the Young function $P$ satisfies the conditions \eqref{doubling cond}, \eqref{growth 1a}, \eqref{growth 2} and \eqref{div condition}, one can choose the sequence $\theta_i$ in such a way that the mapping $F$ belongs to the Orlicz-Sobolev class $W^{1,P}(\Sph^n,\Sph^n)$.
\end{example}

\begin{proof}[Proof of Proposition \ref{prop disc}]
Let $p:\Sph^n \to N$ denote a covering map. This is an orientation preserving (and thus of positive Jacobian) local diffeomorphism, if the universal cover $\tilde{N}$ is diffeomorphic to $\Sph^n$, and a bi-Lipschitz, orientation preserving local homeomorphism, if $\tilde{N}$ is an exotic sphere.

Next, let $\bbbb$ be a Riemannian ball in $M$, diffeomorphic to a Euclidean ball.
There is a Lipschitz map $q:M\to \Sph^n$ which maps $M\setminus\bbbb$ to the south pole $\mathfrak{s}$ of $\Sph^n$ and which is an orientation preserving diffeomorphism between  $\bbbb$ and $\Sph^n\setminus \{\mathfrak{s}\}$; thus $q$ has finite distortion. If $M=\bbbb$, then $q$ is smooth and has everywhere positive Jacobian.
The mapping $G$ is given as the composition
$$
\begin{tikzcd}
M \arrow[rrr,bend left=25,"G"] \arrow[r,"~~~~q"]&\Sph^n\arrow[r,"F"]&\Sph^n\arrow[r,"\!\!\!\!p"]&N.
\end{tikzcd}
$$
Then, if $\mathfrak{n}$ denotes the north pole of $\Sph^n$, $G$ is discontinuous at $x=q^{-1}(\mathfrak{n})$: any neighborhood of $x$ is mapped by $G$ onto the whole $N$. At the same time, $p$ and $q$ are Lipschitz and $F$ is in $W^{1,P}(\Sph^n,\Sph^n)$, thus the composition lies in $W^{1,P}(M,N)$.
\end{proof}

\section{Proof of Theorem~\ref{T5}}
\label{Proof}

Let us choose a Riemannian ball $\bbbb^n\subset M$, diffeomorphic to an $n$-dimensional Euclidean ball and such that $f|_{\partial\bbbb^n}\in W^{1,\alpha}(\partial\bbbb^n)$ and that the image $f(\partial\bbbb^n)$ has small diameter in~$N$. Recall that for any $x\in M$ there exist arbitrarily small balls centered at $x$ with the above property (this follows immediately from Remark~\ref{rem morrey} and Lemma~\ref{lem:Eucl} d)).

By Lemma~\ref{traces} there is an extension $\tilde{g}\in W^{1,n}(\bbbb^n,\R^k)$ of $f|_{\partial\bbbb^n}$.
The extension $\tilde{g}$ can be defined by an explicit integral formula \cite[Theorem~15.21]{leoni}, so that at any point
of $\bbbb^n$, the mapping $\tilde{g}$ is an integral average of values of $f|_{\partial\bbbb^n}$. Hence the image of $\tilde{g}$ is contained in the convex hull of $f(\partial\bbbb^n)$. Since we can assume that the diameter of
$f(\partial\bbbb^n)$ is sufficiently small, we can assume that values of $\tilde{g}$ belong to a tubular neighborhood
$U$ of $N$, from which there is a smooth Lipschitz retraction $\pi:U\to N$. Taking
$g=\pi\circ\tilde{g}\in W^{1,n}(\bbbb^n,N)$ we obtain an extension of $f|_{\partial\bbbb^n}$ with values in $N$.
Clearly, the Jacobian $J_g$ is integrable in~$\bbbb^n$.

Let us thus consider a mapping $F$ from the $n$-dimensional sphere $\bbbs^n$ to $N$, which on the upper hemisphere $\bbbs^n_+$, identified with $\bbbb^n$, equals $f$, and on the lower hemisphere $\bbbs^n_-$, again identified with $\bbbb^n$, equals $g$, see Figure \ref{fig:2}. As $f\in W^{1,P}(\bbbb^n,N)$ and $g\in W^{1,n}(\bbbb^n,N)\subset W^{1,P}(\bbbb^n,N)$, we have that $F\in W^{1,P}(\bbbs^n,N)$.

Next, we apply Lemma~\ref{lemma approx} to the mapping $F$, obtaining a sequence $t_i\to\infty$ and
$Ct_i$\nobreak\mbox{-}\nobreak{}Lipschitz continuous mappings $F_i:\bbbs^n\to N$, $i=1,2,\ldots$ such that
\begin{equation}
\label{eq: 1 for F}
t_i^n |\{\MM |DF|>t_i\}|\to 0\quad\text{ as }i\to\infty,
\end{equation}
and $F_i$ coincides with $F$ on the set $\{\MM|DF|\leq t_i\}$.

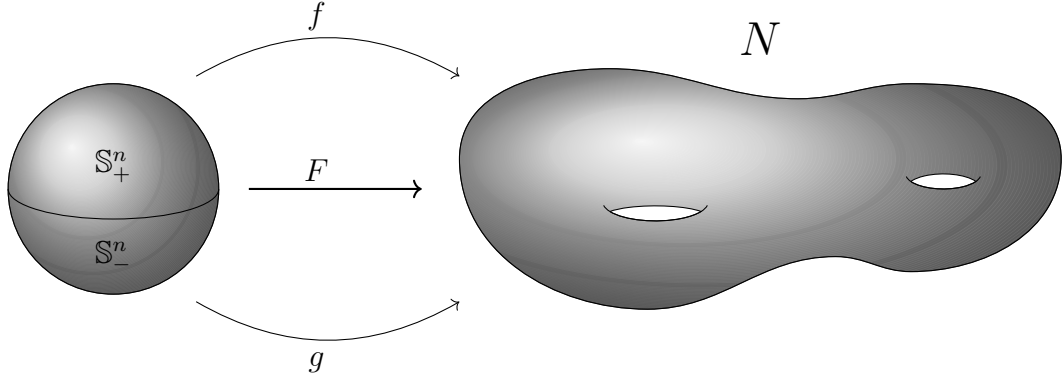
\begin{figure}
\begin{tikzpicture}
%\draw[help lines] (0,-5) grid (15,5);
\shadedraw[shading=ball, ball color=lightgray,opacity=0.6] (1.4,0) circle (1.4);
\draw (1.4,0) circle (1.4);
\draw (0,0) arc [x radius=1.4, y radius =0.4, start angle=-180, end angle=0];

\begin{scope}[shift={(1,0.4)}]

\shadedraw[shading=ball, ball color=lightgray,opacity=0.6] (5,0) to [out=-90, in=180] (7.5,-2) to [out=0, in=180] (10,-1.3) to [out=0,in=180] (11,-1.5) to [out=0,in=-90] (13,0) to [out=90,in=0] (11,1) to [out=180,in=0] (9.5,0.8) to [out=180, in=0] (7,1.2) to [out=180,in=90] (5,0);

\draw  (5,0) to [out=-90, in=180] (7.5,-2) to [out=0, in=180] (10,-1.3) to [out=0,in=180] (11,-1.5) to [out=0,in=-90] (13,0) to [out=90,in=0] (11,1) to [out=180,in=0] (9.5,0.8) to [out=180, in=0] (7,1.2) to [out=180,in=90] (5,0);

\filldraw[fill=white] (7,-0.7) arc [x radius=0.7, y radius=0.25, start angle=210, end angle=330] arc [x radius=0.7, y radius=0.15, start angle=30, end angle=150];
\draw (7,-0.7) arc [x radius=0.7, y radius=0.25, start angle=210, end angle=330, opacity=0] arc [x radius=0.7, y radius=0.25, start angle=330, end angle=350];
\draw (7,-0.7) arc [x radius=0.7, y radius=0.25, start angle=210, end angle=190];

\filldraw[fill=white] (11,-0.3) arc [x radius=0.5, y radius=0.2, start angle=210, end angle=330] arc [x radius=0.5, y radius=0.2, start angle=30, end angle=150];
\draw (11,-0.3) arc [x radius=0.5, y radius=0.2, start angle=210, end angle=330, opacity=0] arc [x radius=0.5, y radius=0.2, start angle=330, end angle=350];
\draw (11,-0.3) arc [x radius=0.5, y radius=0.2, start angle=210, end angle=190];

\end{scope}

\draw[->] (2.5,1.5) to [out=30, in=150] (6,1.5);
\draw[->] (2.5,-1.5) to [out=-30, in=-150] (6,-1.5);
\draw[thick,->] (3.2,0) -- (5.5,0);

\node[scale=1.5] at (10,2) {$N$};
\node at (1.4,0.3) {$\bbbs^n_+$};
\node at (1.4,-0.9) {$\bbbs^n_-$};
\node at (4.1,2.3) {$f$};
\node at (4.1,-2.3) {$g$};
\node at (4.1,0.25) {$F$};
\end{tikzpicture}
\caption{Construction of the mapping $F$.}
\label{fig:2}
\end{figure}

Since, by assumption, the universal cover of the target manifold $N$ is not a rational homo\-logy sphere, by Lemma \ref{lem: RHS} for $i=1,2\ldots$ the mapping $F_i$ has degree zero, thus
$$
\int_{\bbbs^n} J_{F_i}=0.
$$

Therefore
\begin{equation*}
\begin{split}
\int_{\bbbs^n_+} J_{F_i}&=-\int_{\bbbs^n_-} J_{F_i}
=-\int_{\bbbs^n_-\cap \{\MM |DF|\leq t_i\}} J_{g}-\int_{\bbbs^n_-\cap \{\MM |DF|> t_i\}} J_{F_i}\\
&=-\int_{\bbbs^n_-} J_{g}+\int_{\bbbs^n_-\cap \{\MM |DF|> t_i\}} J_{g}-\int_{\bbbs^n_-\cap \{\MM |DF|> t_i\}} J_{F_i}\\
&=\left(-\int_{\bbbs^n_-} J_{g}\right)+I_i-K_i.
\end{split}
\end{equation*}
It follows from \eqref{eq: 1 for F} that
the measure of the set of integration in $I_i$ tends to 0 when $i\to\infty$. This, together with the fact that $J_g$ is an integrable function, shows that $I_i\to 0$ with $i\to\infty$.

Condition \eqref{eq: 1 for F} allows us to prove the same for $K_i$. Namely, since $F_i$ is Lipschitz with the Lipschitz constant $Ct_i$,
we have $|J_{F_i}|\leq Ct_i^n$ and
$$
|K_i|\leq C t_i^n |\bbbs^n_-\cap \{\MM |DF|> t_i\}|\leq C t_i^n |\{\MM |DF|> t_i\}|\xrightarrow{i\to\infty} 0.
$$

Hence
\begin{equation}
\label{lala}
\int_{\bbbs^n_+} J_{F_i}\xrightarrow{i\to\infty} -\int_{\bbbs^n_-} J_{g}<\infty.
\end{equation}
However,
\begin{equation}
\label{baba}
\int_{\bbbs^n_+} J_{F_i}=\int_{\bbbs^n_+\cap \{\MM |DF|\leq t_i\}} J_f+\int_{\bbbs^n_+\cap \{\MM |DF_i|>t_i\}} J_{F_i}=I'_i+K'_i\to \int_{\Sph^n_+} J_f.
\end{equation}
Indeed, $J_f\geq 0$, and the measure of the complement of the set of integration in $I'_i$ tends to zero,
$|\bbbs^n_+\cap \{\MM |DF_i|>t_i\}|\xrightarrow{i\to \infty}0$,
so the Monotone Convergence Theorem yields
$I'_i\to \int_{\bbbs^n_+} J_f$ as $i\to\infty$. Also, by the same argument as when estimating $K_i$,
$$
|K'_i|\leq Ct_i^n |\{\MM |DF|>t_i\}|\xrightarrow{i\to\infty}0.
$$
Comparing \eqref{lala} and \eqref{baba} we obtain
$$
\int_{\bbbb^n} J_f=-\int_{\bbbs^n_-} J_{g}<\infty
$$
which proves integrability of $J_f$ in $\bbbb^n$.
\hfill $\Box$

\section{Proof of Theorem~\ref{T4}}
\label{Michal}

The key step  of the proof
is the following lemma, which is essentially \cite[Proposition 7]{GHP}, under weaker assumptions.
\begin{lemma}
\label{thm main}
Assume that $M$ and $N$ are smooth, oriented, $n$-dimensional Riemannian manifolds
without boundary and assume additionally that $N$ is compact.
If $P$ satisfies conditions
\eqref{doubling cond}, \eqref{growth 1b}, \eqref{growth 2}, \eqref{div condition} and
$f\in W^{1,P}(M,N)$ has finite distortion and locally integrable Jacobian, $J_f\in L^1_{\rm loc}$, then for every point $x\in N$ and every $\eps>0$ there is $r>0$ such that
$\diam f(\bbbb^n(x,r))<\eps$.
\end{lemma}
\begin{remark}
To be more precise, we will show that almost all points of the ball $\bbbb^n(x,r)$ are mapped on a set of diameter less than $\eps$, but then we can find a representative of $f$ such that $\diam f(\bbbb^n(x,r))<\eps$.
\end{remark}

By Lemma~\ref{thm main}, we can find $r>0$ such that
both $\bbbb^n(x,r)$ and $f(\bbbb^n(x,r))$ lie then in local charts, so the problem is reduced to the Euclidean one and the result follows from Theorem~\ref{T2}.

In the proof of Lemma~\ref{thm main} we need a simple geometric lemma.
\begin{lemma}
\label{lem:retr}
Assume $N$ is a connected Riemannian manifold.  Let $D$, $D'$ be Riemannian balls in $N$
with disjoint closures and such that balls that are concentric with $D$ and $D'$ and with twice the radius are diffeomorphic to
$n$-dimensional Euclidean balls. Then there is a Lipschitz retraction $\pi:N\setminus \overbar{D}' \to \overbar{D}$.
\end{lemma}
For a detailed construction of such a retraction see \cite[Proof of Proposition~7]{GHP}.

\begin{proof}[Proof of Lemma \ref{thm main}]
Let $\eps>0$ be so small that
$$
\delta=\inf_{\substack{D\subset N\\ \diam D<\eps}} |N\setminus D|>0,
$$
where the infimum is taken over all balls $D\subset N$ of diameter less than $\eps$.

By Lemma~\ref{lem:Eucl} d), for every $x\in M$ there is a sufficiently small $r>0$ such that
$\bbbb^n=\bbbb^n(x,r)$ satisfies
\begin{itemize}
\item
$f|_{\partial\bbbb^n}\in W^{1,\alpha}(\partial\bbbb^n)$,
\item
$f(\partial\bbbb^n)$ is contained in a ball $D\subset N$ of diameter less than $\eps$,
\item
$\int_{\bbbb^n} J_f<\delta$.
\end{itemize}

To complete the proof it suffices to show that almost all points of $\bbbb^n$ are mapped into $\overbar{D}$, i.e.
$$
|A|=0,
\quad
\text{where}
\quad
A=f^{-1}(N\setminus \overbar{D})\cap\bbbb^n.
$$
Suppose to the contrary that $|A|>0$. We claim that
\begin{equation}
\label{jac pos}
\int_A J_f>0.
\end{equation}
Consider the function $h:\bbbb^n\to \bbbr$ defined by $h(x)=\dist(f(x),\overbar{D})$.
Since the distance function $\dist(\cdot,\overbar{D})$ is $1$-Lipschitz, $h\in W^{1,P}(\bbbb^n)\subset W^{1,1}(\bbbb^n)$
and hence $h\in W^{1,1}_0(\bbbb^n)$, because the trace of $h$ on the boundary $\partial\bbbb^n$ equals zero.
Clearly $Dh=0$ on $\bbbb^n\setminus A$, because $h=0$ on that set. On the other hand $h$ is not constant since
$h>0$ on the set $A$ of positive measure so the derivative $Dh$ cannot be equal zero a.e.
in $\bbbb^n$ (as otherwise we would have $h=0$ in $\bbbb^n$) so $Dh$ must be non-zero
on a subset of $A$ of positive measure. Since the mapping $f$ has finite distortion, \eqref{jac pos} follows.

By $\bbbb^n_\sigma$, for $\sigma>0$, we shall denote a ball concentric with $\bbbb^n$, but of radius $\sigma$ times that of $\bbbb^n$.

Using Lemma~\ref{traces}
we extend the H\"older continuous (by Lemma~\ref{morrey})
function $f|_{\partial\bbbb^n}$ to a $W^{1,n}$ function $\tf$ on an annulus
$\bbbb^n_{1+2\delta}\setminus\bbbb^n$, for some small $\delta>0$. We can choose the extension $\tf$ to be smooth in the annulus and continuous up to the boundary
(see Remark~\ref{szesnascie}). Hence if $\delta$ is sufficiently small, values of $\tf$ on the annulus $\bbbb^n_{1+2\delta}\setminus\bbbb^n$
belong to a tubular neighborhood of $N$ and composition with the nearest point projection shows that we can assume that the values of the extension $\tf$ belong to $N$.
Then, if $\delta$ is small enough, continuity of $\tf$ shows that
$\tf( \bbbb^n_{1+2\delta}\setminus \bbbb^n)\subset D$.

Let
$$
g=\begin{cases}
\tf &\text{on }\bbbb^n_{1+2\delta}\setminus \bbbb^n,\\
f&\text{on }\bbbb^n.
\end{cases}
$$
Recall that
$$
\int_{\bbbb^n} J_f<\delta\leq |N\setminus D|.
$$
Since the extension $\tilde{f}$ belongs to $W^{1,n}$ and $W^{1,n}$ mappings have integrable Jacobians, we can take $\delta>0$ so small that
\begin{equation}
\label{ucho}
\int_{\bbbb^n_{1+2\delta}\setminus \bbbb^n} |J_g|<\int_A J_f
\quad
\text{and}
\quad
\int_{\bbbb^n_{1+2\delta}}|J_g|<|N\setminus D|.
\end{equation}
Note that $g\in W^{1,P}(\bbbb^n_{1+2\delta},N)$.
Applying Lemma~\ref{lemma approx} to $g$ defined on the manifold (with boundary) $\bbbb^n_{1+2\delta}$
we obtain a sequence $t_i\to\infty$ and $Ct_i$-Lipschitz mappings $g_i$ such that $g_i$ coincides with $g$ on $\bbbb^n_{1+2\delta}\cap \{\MM|Dg|\leq t_i\}$.

The function $x\mapsto \MM |Dg|(x)$ is bounded on $\partial \bbbb^n_{1+\delta}$: since $g$ is smooth in the annulus
$\bbbb^n_{1+2\delta}\setminus \bbbb^n$, averages of $|Dg|$ on small balls centered at $\partial \bbbb^n_{1+\delta}$ are uniformly bounded while
averages on larger balls are bounded by $C\int_{\bbbb^n_{1+2\delta}}|Dg|$.
Thus we can choose $i$ large enough to have
$\partial \bbbb^n_{1+\delta}\subset \{\MM |Dg|\leq t_i\}$. Then
$g_i=g$ on $\partial \bbbb^n_{1+\delta}$, in particular
$g_i(\partial \bbbb^n_{1+\delta})=\tf(\partial \bbbb^n_{1+\delta})\subset D$.

Next, we see that, by Lemma~\ref{lemma approx},
\begin{equation*}
|g_i(\overbar{\bbbb}^n_{1+\delta})|\leq \int_{\bbbb^n_{1+\delta}}|J_{g_i}|=
\int_{\bbbb^n_{1+\delta}}|J_{g}|+\int_{\bbbb^n_{1+\delta}\cap \{\MM |Dg|>t_i\}} |J_{g_i}|-|J_{g}|\xrightarrow{i\to\infty}
\int_{\bbbb^n_{1+\delta}}|J_g|<|N\setminus D|,
\end{equation*}
since $|J_g|$ is integrable and
$$
\Big|\int_{\bbbb^n_{1+\delta}\cap \{\MM |Dg|>t_i\}} |J_{g_i}|-|J_{g}|\Big|
\leq Ct_i^n |\{\MM |Dg|>t_i\}|\,\,+\int_{\bbbb^n_{1+\delta}\cap \{\MM |Dg|>t_i\}} |J_{g}|\xrightarrow{i\to\infty}0.
$$
Therefore,
for $i$ sufficiently large we have
$|g_i(\overbar{\bbbb}^n_{1+\delta})|<|N\setminus D|$, and since the set $g_i(\overbar{\bbbb}^n_{1+\delta})$ is closed in $N$, there exists a
Riemannian ball $D'\subset N\setminus\overbar{D}$ such that $D'\cap g_i(\overbar{\bbbb}^n_{1+\delta})=\varnothing$.

Let $\pi:N\setminus \overbar{D}'\to \overbar{D}$ be the Lipschitz retraction given by Lemma \ref{lem:retr}.
Then $g_i|_{\partial \bbbb^n_{1+\delta}}=\pi\circ g_i|_{\partial \bbbb^n_{1+\delta}}$, because $g_i(\partial \bbbb^n_{1+\delta})\subset D$.
Now, we repeat the construction as in Figure \ref{fig:2} and in the proof of Theorem~\ref{T5}:
we construct a Lipschitz mapping $G_i$ from a sphere into $N$, that equals $g_i|_{\overbar{\bbbb}^n_{1+\delta}}$
on the northern hemisphere and $\pi\circ g_i|_{\overbar{\bbbb}^n_{1+\delta}}$ on the southern hemisphere.
Since the mapping $G_i:\Sph^n\to N$ is not surjective (it omits the ball $D'$ in $N$), its degree equals zero, thus
\begin{equation}
\int_{\bbbb^n_{1+\delta}}J_{g_i}=\int_{\bbbb^n_{1+\delta}}J_{\pi\circ g_i}.
\end{equation}

On the other hand,
since $\pi\circ g_i$ maps $\bbbb^n_{1+\delta}\cap g_i^{-1}(N\setminus \overbar{D})$ onto the boundary of $\overbar{D}$ that has dimension $n-1$,
the Jacobian $J_{\pi\circ g_i}$
equals zero on that set, so
$$
J_{\pi\circ g_i}=\begin{cases} J_{g_i}&\text{ in } \bbbb^n_{1+\delta}\cap g_i^{-1}(\overbar{D}),\\
0 &\text{ in }\bbbb^n_{1+\delta}\cap g_i^{-1}(N\setminus \overbar{D}).
\end{cases}
$$
Thus
\begin{equation}
\label{eq:contr}
\int_{\bbbb^n_{1+\delta}\cap g_i^{-1}(N\setminus \overbar{D})}J_{g_i}=0.
\end{equation}
However,
\begin{equation*}
\begin{split}
\int_{\bbbb^n_{1+\delta}\cap g_i^{-1}(N\setminus \overbar{D})}J_{g_i}
&=\int_{\bbbb^n_{1+\delta}\cap g_i^{-1}(N\setminus \overbar{D})\cap \{\MM|Dg|>t_i\}}J_{g_i}+
\int_{\bbbb^n_{1+\delta}\cap g^{-1}(N\setminus \overbar{D})\cap \{\MM|Dg|\leq t_i\}}J_{g}\\
&=S_i+T_i.
\end{split}
\end{equation*}
We can estimate $|S_i|$ by $Ct_i^n |\{\MM|Dg|>t_i\}|$, thus $S_i\to 0$ as $i\to\infty$. The estimate for $T_i$ is more difficult.
We have
\begin{equation*}
\begin{split}
T_i&=
\int_{\bbbb^n\cap f^{-1}(N\setminus\overbar{D})\cap \{\MM|Dg|\leq t_i\}} J_f +
\int_{(\bbbb^n_{1+\delta}\setminus\bbbb^n)\cap g^{-1}(N\setminus \overbar{D})\cap \{\MM|Dg|\leq t_i\}} J_g\\
&\geq
\int_{A\cap\{\MM|Dg|\leq t_i\}} J_f - \int_{\bbbb^n_{1+\delta}\setminus\bbbb^n} |J_g|
\xrightarrow{i\to\infty}
\int_A J_f - \int_{\bbbb^n_{1+\delta}\setminus\bbbb^n} |J_g|\stackrel{\eqref{jac pos},\eqref{ucho}}{>}0
\end{split}
\end{equation*}
which contradicts \eqref{eq:contr}.

This concludes the proof of Lemma~\ref{thm main} and hence that of Theorem~\ref{T4}.
\end{proof}

\end{document}